\documentclass[11pt, reqno]{amsart}

\usepackage[colorlinks]{hyperref}
\usepackage{amssymb}
\usepackage{mathrsfs}
\usepackage[all,cmtip]{xy}
\usepackage{cite}
\usepackage{color}
\usepackage{tikz-cd}
\usepackage{framed}
\usepackage{enumitem}

\newtheorem{theorem}{Theorem}

\newtheorem{lemma}[theorem]{Lemma}
\newtheorem{proposition}[theorem]{Proposition}

\theoremstyle{definition}

\counterwithout{equation}{section}
\newtheorem*{Brown-Eagleson Theorem}{The Brown-Eagleson Theorem}

\newcommand{\D}{\mathbb D}
\newcommand{\R}{\mathbb R}
\newcommand{\C}{\mathbb C}
\newcommand{\F}{\mathscr F}
\newcommand{\T}{\mathbb T}
\newcommand{\E}{\mathbb E}
\newcommand{\PP}{\mathbb P}
\newcommand{\cal}{\mathcal}
\newcommand{\ol}{\overline}

\newcommand{\la}{\langle}
\newcommand{\ra}{\rangle}

\allowdisplaybreaks

\title{Normal approximation for iterated inner functions}

\author{Yukun Chen}
\address{Yukun Chen: School of Mathematics and Statistics, Wuhan University, Wuhan 430072, China}
\email{yukunchen@whu.edu.cn}

\author{Xiangdi Fu}
\address{Xiangdi Fu: School of Fundamental Physics and Mathematical Sciences, HIAS, University of Chinese Academy of Sciences, Hangzhou, 310024, China}
\email{xdfu@ucas.ac.cn}

\author{Zhaofeng Lin}
\address{Zhaofeng Lin: School of Fundamental Physics and Mathematical Sciences, HIAS, University of Chinese Academy of Sciences, Hangzhou, 310024, China}
\email{linzhaofeng@ucas.ac.cn}

\author{Yanqi Qiu}
\address{Yanqi Qiu: School of Fundamental Physics and Mathematical Sciences, HIAS, University of Chinese Academy of Sciences, Hangzhou, 310024, China}
\email{yanqiqiu@ucas.ac.cn}


\makeatletter
\@namedef{subjclassname@2020}{\textup{2020} Mathematics Subject Classification}
\makeatother

\begin{document}
	\begin{abstract}
	 A Berry--Ess\'{e}en theorem for linear combinations of iterates of an inner function is obtained. Our proof, which is based an elementary transfer argument and classical results in martingale theory, also leads to a simple proof of Nicolau and Soler i Gibert's central limit theorem for inner functions.
	\end{abstract}
	\maketitle
	

	\section{Introduction}
	Let $\C$ be the complex plane, and let $\D= \{z\in \C: |z|<1\}$ be the open unit disc. Let $\T=\{z\in \C: |z|=1\}$ be the unit circle, equipped with the normalized Lebesgue measure $m$. An analytic function $f:\D \to \D$ is called {\it inner} if its radial limits $$f(\omega):= \lim_{r\to 1} f(r\omega)$$ have modulus one at almost every $\omega\in \T$. Therefore, by taking radial limits, any inner function $f$ induces a boundary map from $\T$ into itself defined at almost every point. This boundary map will also be denoted by $f$. 
	
	The study of inner functions is a fundamental topic in complex analysis, and a comprehensive account of their properties and related theories can be found in classical monographs; see \cite{Gar81} for example. It is well known that if $f(0)=0$, then the corresponding boundary map preserves the normalized Lebesgue measure $m$, meaning $m(f^{-1}(E)) = m (E)$ for all measurable set $E\subset \T$. This measure-preserving property has motivated extensive research into their dynamical behavior \cite{DM91,Pom81}, which has in turn drawn considerable attention to the iteration of inner functions \cite{Nic22,NS22,NS24}.

	Throughout this note, we denote by $f^{\circ n}$ the $n$-th iterate of $f$, and by $\mathcal{N}(0,\sigma^2)$ the real normal distribution with mean zero and variance $\sigma^2$. We also denote by $\mathcal{CN}(0,\sigma^2)$ the circularly symmetric complex normal distribution with mean zero and variance $\sigma^2$ (equivalently, its real and imaginary parts are independent $\mathcal{N}(0,\sigma^2/2)$ variables).

	The main result of this note is the following Berry--Ess\'{e}en type theorem for iterated inner functions.

\begin{theorem}\label{thm-BE-inner}
	Let $f$ be an inner function with $f(0)=0$ which is not a rotation. Let $(a_n)_{n=1}^\infty$ be a sequence of complex numbers satisfying
	\begin{align}\label{eq-Lindeberg-cond}\sum_{n=1}^\infty |a_n|^2 = \infty,\quad \text{and} \quad \lim_{N\to \infty} \frac{|a_N|^2}{\sum_{n=1}^N |a_n|^2} = 0,
	\end{align} 
	and set $$Z_N:= \frac{1}{\sigma_N}\sum_{n=1}^N a_n f^{\circ n},\quad \text{and }\,\, \sigma_N:= \bigg\| \sum_{n=1}^N a_n f^{\circ n }\bigg\|_{L^2(\T)}.$$ Then for any $\delta\geq 1$, there exist positive constants $C(\delta,f)$ (depending only on $\delta$ and $f$), such that the following Cram\'er-Wold error estimate holds: 
		\begin{equation}
			\begin{split}
		&\sup_{x\in \R,\, \alpha \in \T} \bigg| \PP\Big( {\rm Re}\, \big( \alpha Z_N\big) \leq x \Big) - \PP\Big( \cal{N}(0,1/2) \leq x \Big)\bigg|  \\ 
		& \,\leq \,C(\delta,f) \cdot \bigg(\sum_{n=1}^N |a_n|^{4} \bigg)^{\frac{1+\delta}{6+4\delta}}  \bigg(\sum_{n=1}^N |a_n|^{2} \bigg)^{-\frac{1+\delta}{3+2\delta}} + 2.7 \cdot \frac{|\lambda|}{1-|\lambda|} \frac{\big|\sum_{n=1}^N \lambda^{N-n} a_n\big|}{\sqrt{\sum_{n=1}^N |a_n|^2}}.
			\end{split}\label{eq-rate-an}
		\end{equation}
	\end{theorem}

	We point out that Theorem \ref{thm-BE-inner} provides a refinement of the Central Limit Theorem (CLT) for inner functions (Theorem \ref{thm-NS} below) recently established by A. Nicolau and O. Soler i Gibert \cite{NS22,NS24}. Indeed, if the sequence \((a_n)_{n=1}^\infty \) satisfies the conditions in \eqref{eq-Lindeberg-cond}, then the quantity in \eqref{eq-rate-an} tends to zero for any fixed \(\delta \geq 1\), thereby establishing the convergence in distribution: $$\frac{1}{\sigma_N} \sum_{n=1}^N a_n f^{\circ n} \stackrel{d}{\longrightarrow} \cal{CN}(0,1).$$ Moreover, in the constant coefficient case $a_n \equiv 1$, the bound \eqref{eq-rate-an} provides a decay rate of $O(N^{-\frac{1}{4} + \varepsilon})$ for any $\varepsilon > 0$. An interesting problem is to determine the optimal rate of decay of the errors arising in the convergence in distribution in the central limit theorem for inner functions. 
	

	The proof of Theorem \ref{thm-BE-inner} is based on an elementary transfer argument together with classical results in martingale theory. In particular, our method also yields a short proof of Nicolau and Soler i Gibert's CLT. 

	\noindent\textbf{The structure of this paper:} Since the martingale CLT is considerably more elementary and more familiar to a general audience than its Berry--Ess\'{e}en counterpart, we prefer to first present our short proof of Nicolau and Soler i Gibert's CLT, and then demonstrate how a simple modification of the argument leads to Theorem \ref{thm-BE-inner}.

	\section{A simple proof of the central limit theorem for inner functions}

	In this section, we present a simple proof of the following

		\begin{theorem}[Nicolau and Soler i Gibert]\label{thm-NS}
		Let $f$ be an inner function with $f(0)=0$ which is not a rotation. Let $(a_n)_{n=1}^\infty$ be a sequence of complex numbers satisfying the conditions in \eqref{eq-Lindeberg-cond}, and set $$\sigma_N:= \bigg\| \sum_{n=1}^N a_n f^{\circ n }\bigg\|_{L^2(\T)}.$$ Then the following convergence in distribution holds: 
		\begin{align}\label{eq-convergence-distribution}\frac{1}{\sigma_N} \sum_{n=1}^N a_n f^{\circ n }\stackrel{d}{\longrightarrow} \cal{CN}(0,1).
		\end{align}
	\end{theorem}

	We note that the normal approximation \eqref{eq-convergence-distribution} was first established in \cite{NS22} under certain conditions on the size of the coefficient sequence $(a_n)_{n=1}^\infty$, and subsequently refined in \cite{NS24} to the sharp form stated in Theorem \ref{thm-NS}. The original proof relies on the theory of Aleksandrov-Clark measures as the main technical tool to establish the pointwise convergence of the characteristic functions.

	\subsection{The reverse martingale differences}
	
	Let $f$ be an inner function with $f(0)=0$, which is not a rotation. We denote by $\F_0$ the $\sigma$-algebra consisting of all Lebesgue measurable sets on $\mathbb{T}$, and let $\F_n$ be the complete sub-$\sigma$-algebra of $\F_0$ generated by $f^{\circ n}$. Clearly, the sequence $(\F_n)_{n=0}^\infty$ is decreasing: $$\F_0 \supset  \F_1 \supset \F_2 \supset \cdots.$$ To simplify notation, we set $\lambda := \overline{f'(0)}$ and $\mu:=\ol{f''(0)}/2$. Note that (by the Schwarz--Pick lemma)  $$|\lambda|<1, \quad \text{and } \, |\mu|\leq 1-|\lambda|^2.$$
	
	Now we define 
	\begin{align}\label{eq-def-Yn}
		Y_n:= f^{\circ n}-\lambda f^{\circ (n+1)}.
	\end{align} It turns out that $(Y_n)_{n=1}^\infty$ naturally forms a \emph{reverse} martingale difference sequence. The details are presented in the following lemma.

	\begin{lemma}\label{lem-Yn}
		The sequence $(Y_n)_{n=1}^\infty$ forms a reverse martingale difference sequence with respect to $(\F_n)_{n=0}^\infty$, that is, $$Y_n \text{ is $\F_n$-measurable}, \,\, \text{and }\,\, \E[Y_n|\F_{n+1}] = 0.$$
		Moreover, we have the following identities:
		\begin{enumerate}[label=(\roman*),font=\normalfont]
			\item \label{item:|Yn|2} $\E\big[|Y_n|^2 | \F_{n+1}\big]=\E\big[|Y_n|^2\big] = 1-|\lambda|^2.$
			\item \label{item:Yn2} $\E\big[Y_n^2| \F_{n+1}\big]=\mu f^{\circ (n+1)}$.
			\item \label{item:reYn2} For any $\alpha \in \mathbb C$, $$\E\big[({\rm Re}\, \alpha Y_n)^2 |\mathscr F_{n+1}\big]=\frac{1-|\lambda|^2}{2}|\alpha|^2 + \frac{{\rm Re}\,(\alpha^2 \mu f^{\circ (n+1)})}{2}.$$
		\end{enumerate}
		
	\end{lemma}
	
		The proof of Lemma \ref{lem-Yn} requires only the following standard result: if $\eta$ is an inner function vanishing at the origin, then the collection $$\{\eta^k : k\in \mathbb Z \}$$ constitutes an orthonormal basis for $L^2(\F(\eta)):=\{f\in L^2(\T): \text{$f$ is $\F(\eta)$-measurable}\}$, where $\F(\eta)$ is the complete sub $\sigma$-algebra generated by $\eta$. Since the conditional expectation $\E\big[ \cdot | \F(\eta)\big]$ acts as the orthogonal projection from  $L^2(\T)$ onto $L^2(\F(\eta))$, it holds that $$\E\big[ h| \F(\eta)\big] = \sum_{k\in \mathbb Z} \, \big\la h, \eta^k \big\ra \cdot \eta^k,\quad h\in L^2.$$ The rest of the proof of Lemma \ref{lem-Yn} consists of standard calculations. To keep the exposition focused, we have postponed it (and the proof of Lemma \ref{lem-transfer} below) to the end of this note.

		\subsection{The transfer argument}
		
		The second key ingredient is the following elementary identity, which rewrites a linear combination of iterates of inner functions as a linear combination of reverse martingale differences plus a remainder term.

	\begin{proposition}\label{lem-inner=Yn+remainder}
	For any complex sequence $(a_n)_{n=1}^\infty$, and any positive integer $N$, it holds that: 
		\begin{align}\label{eq-sum-to-martingale}
			\sum_{n=1}^N a_n f^{\circ n}=\sum_{n=1}^N b_n Y_n + \lambda b_N f^{\circ (N+1)}, 
		\end{align}
		where the coefficients $b_n$ are given by 
		\begin{align}\label{eq-def-bn}
			b_n=\sum_{k=1}^n \lambda^{n-k} a_k, \quad n\geq 1.
		\end{align}
	\end{proposition}

One can easily show that if the original coefficients satisfy the conditions in \eqref{eq-Lindeberg-cond}, then so do the resulting coefficients:
	
	\begin{lemma}\label{lem-transfer}
		Let $a=(a_n)_{n=1}^\infty$ be any complex sequence, and let $b=(b_n)_{n=1}^\infty$ be the associated sequence defined as in \eqref{eq-def-bn}. Then the following claims hold true: 
		\begin{enumerate}[label=(\roman*),font=\normalfont]
			\item \label{item:keep-Lindberg}If $$\lim_{N\to \infty} \frac{\max_{1\leq n \leq N} |a_n|^2}{\sum_{n=1}^N |a_n|^2} = 0,$$ then $$\lim_{N\to \infty} \frac{\max_{1\leq n \leq N} |b_n|^2}{\sum_{n=1}^N |b_n|^2} =0.$$
			\item \label{item:keep-no-ell2} If $a\notin \ell^2$, then $b\notin \ell^2$.
			\item \label{item:keep-ell2}If $a\in \ell^2$, then $b\in \ell^2$ and $$\frac{1}{1+|\lambda|}\cdot \|a\|_{\ell^2} \leq \|b\|_{\ell^2} \leq \frac{1}{1-|\lambda|}\cdot \|a\|_{\ell^2}.$$
		\end{enumerate}
	\end{lemma}

	Note that if $a\notin \ell^2$, then $$\lim_{N\to \infty} \frac{|a_N|^2}{\sum_{n=1}^N |a_n|^2}=0 \text{\, if and only if \,}\lim_{N\to \infty} \frac{\max_{1\leq n\leq N}|a_n|^2}{\sum_{n=1}^N |a_n|^2}=0.$$ Moreover, in the case $a\in \ell^2$, one can establish the following inequalities:
	\begin{align}\label{eq-2-norm-eqvalence}
		\frac{1-|\lambda|}{1+|\lambda|} \cdot \|a\|^2_{\ell^2} \leq \bigg\|\sum_{n=1}^\infty a_n f^{\circ n}\bigg\|^2_2 \leq \frac{1+|\lambda|}{1-|\lambda|}\cdot \|a\|^2_{\ell^2}.
	\end{align}
	Indeed, for any $a\in \ell^2$, the associated sequence $b$ defined by \eqref{eq-def-bn} is also square summable. Letting $N\to \infty$ in \eqref{eq-sum-to-martingale} yields $$\sum_{n=1}^\infty a_n f^{\circ n} = \sum_{n=1}^\infty b_n Y_n.$$ Since the reverse martingale differences $Y_n$ are pairwise orthogonal and $\E\big[ |Y_n|^2\big] = 1-|\lambda|^2$, we conclude that $$\bigg\|\sum_{n=1}^\infty a_n f^{\circ n}\bigg\|^2_2 = \big( 1- |\lambda|^2\big) \cdot \|b\|^2_2.$$ Combining this with Lemma \ref{lem-transfer}-\ref{item:keep-ell2}, we obtain the inequalities in \eqref{eq-2-norm-eqvalence}. Moreover, since $Y_n \perp f^{\circ (N+1)}$ for each $1\leq n \leq N$, it follows from \eqref{eq-sum-to-martingale} that 
	\begin{align}\label{eq-rhoN-sigmaN} 
		\bigg\|\sum_{n=1}^N a_n f^{\circ n}\bigg\|^2_2 = (1-|\lambda|^2)\sum_{n=1}^N |b_n|^2  + |\lambda|^2 |b_N|^2.
	\end{align}
	
	It should be pointed out that \eqref{eq-2-norm-eqvalence} was established in \cite[Theorem 9]{NS22} through properties of Toeplitz matrices. 
	
	\subsection{A weak law of large numbers for inner functions}

	\begin{lemma}\label{lem-weak-law}
		Let $f$ be an inner function with $f(0)=0$ which is not a rotation. Suppose the complex sequence $(a_n)_{n=1}^\infty$ satisfies $$ S_N:=\sum_{n=1}^N |a_n| \to \infty, \quad \text{and}\quad  \frac{|a_N|}{S_N}\to 0.$$ Then $\frac{1}{S_N} \sum_{n=1}^N a_n f^{\circ n} \stackrel{p}{\longrightarrow} 0$, i.e.  converges to zero in probability. 
	\end{lemma}
	\begin{proof} 
		We only need to prove $\frac{1}{S_N} \sum_{n=1}^N a_n f^{\circ n}$ converges to zero in $L^2$. In view of \eqref{eq-2-norm-eqvalence}, it suffices to prove $$\lim_{N\to \infty} \frac{\sum_{n=1}^N |a_n|^2}{S^2_N}=0.$$ Without loss of generality, we may assume $a_n\neq 0$ for each $n\geq 1$ and $\sum_{n=1}^\infty |a_n|^2 = \infty$. Then it follows from the Stolz--Ces\`aro theorem that $$\limsup_{N\to \infty} \frac{\sum_{n=1}^N |a_n|^2}{S^2_N} \leq  \limsup_{N\to \infty} \frac{|a_N|^2}{S_N^2-S^2_{N-1}}= \limsup_{N\to \infty}\frac{|a_N|}{ S_N+S_{N-1}}=0.$$
		The proof is completed. 
	\end{proof}

	\subsection{A central limit theorem for martingale triangular arrays}

	 We need to employ a classical result due to B. M. Brown and G. K. Eagleson \cite[Corollary 1]{BE71}. For the reader's convenience, we provide its precise statement here.

	\begin{Brown-Eagleson Theorem}\label{thm-BE71}
		Consider the following triangular array of real-valued random variables on a probability space $(\Omega, \cal F, \mathbb P)$:
		\begin{align}\label{eq-martingale-array}
		\begin{pmatrix}
 			X_{1,1} & 0 & 0 & \cdots \\
			X_{2,1} & X_{2,2} & 0 & \cdots \\
			X_{3,1} & X_{3,2} & X_{3,3} & \cdots \\
			\vdots & \vdots & \vdots & \ddots
		\end{pmatrix}
	\end{align}
		Let $\{\mathcal{F}_{N,n}: 0\leq n \leq N\}$ be a family of $\sigma$-algebras such that for each $N\geq 1$ fixed, 
		\begin{align}\label{eq-condition-incresing-chain}
			\cal F_{N,0} \subset \cal F_{N,1} \subset \cdots \subset \cal F_{N,N} \subset \cal F, 
		\end{align} 
		and for any $1\leq n\leq N$, 
		\begin{align}\label{eq-condition-martingale} 
			X_{N,n} \text{ is $\cal F_{N,n}$-measurable}, \quad \E\big[ X_{N,n} | \cal F_{N,n-1}\big]=0.
		\end{align}
		Define
		$$
			V^2_N:= \sum_{n=1}^N \E\big[X_{N,n}^2| \cal F_{N,n-1}\big],\quad  m_N:= \max_{1\leq n\leq N} \E\big[X_{N,n}^2| \cal F_{N,n-1}\big].
		$$ 
		If the following three conditions are satisfied as $N\to \infty$: 
		\begin{itemize}
			\item [(C1)] $m_N \stackrel{p}{\longrightarrow} 0$;
			\item [(C2)] $V^2_N \stackrel{p}{\longrightarrow} \sigma^2$ (constant);
			\item [(C3)] For any $\varepsilon>0$, $$\sum_{n=1}^N \E[X_{N,n}^2 I\big(|X_{N,n}| \geq \varepsilon\big)| \cal F_{N,n-1}] \stackrel{p}{\longrightarrow}  0,$$
		\end{itemize}
		 then the $N$-th row sum of \eqref{eq-martingale-array} converges in law to $\cal N(0,\sigma^2)$ as $N\to \infty$: $$X_{N,1}+ \cdots + X_{N,n}\stackrel{d}{\longrightarrow} \cal N(0,\sigma^2).$$
	\end{Brown-Eagleson Theorem}

	For further generalizations of this result, we refer the reader to \cite{Eag75}.

	\subsection{Concluding the proof}

	\begin{proof}[Proof of Theorem \ref{thm-NS}]
		Let $Y_n$ and $b_n$ be defined as in \eqref{eq-def-Yn} and \eqref{eq-def-bn}, respectively. Then by Lemma \ref{lem-transfer}, 
		\begin{align}\label{eq-bn-Lindberg}
			\sum_{n=1}^\infty |b_n|^2=\infty,\quad {\rm and} \quad \lim_{N\to \infty} \frac{|b_N|^2}{\sum_{n=1}^N |b_n|^2}= 0.
		\end{align}
		Combining this with Lemma \ref{lem-Yn}-\ref{item:|Yn|2} and \eqref{eq-sum-to-martingale}, one has $$\lim_{N\to \infty} \frac{\sigma^2_N}{\rho_N^2} = 1-|\lambda|^2, \quad {\rm where}\,\, \rho_N:= \sqrt{\sum_{n=1}^N |b_n|^2}.$$ In virtue of the identity \eqref{eq-sum-to-martingale}, to prove $$\frac{1}{\sigma_N}\sum_{n=1}^N a_nf^{\circ n}\stackrel{d}{\longrightarrow}\cal{CN}(0,1),$$ it suffices to show
		$$\frac{1}{\rho_N}\sum_{n=1}^N b_n Y_n \stackrel{d}{\longrightarrow} \mathcal{CN}(0,1-|\lambda|^2).$$ Furthermore, by circular symmetry, we only need to establish that $${\rm Re}\,\bigg(\frac{\alpha}{\rho_N}\sum_{n=1}^N b_n Y_n\bigg) \stackrel{d}{\longrightarrow} \mathcal N\bigg(0,\frac{1-|\lambda|^2}{2}\bigg),$$ for any fixed $\alpha\in \T$.

		Now we set $$X_{N,n}:= {\rm Re}\, \bigg( \frac{\alpha \,b_{N+1-n} Y_{N+1-n}}{\rho_N}\bigg),\quad 1\leq n\leq N,$$ and $$\mathcal{F}_{N,n} := \mathscr{F}_{N-n+1},\quad 0\leq n \leq N.$$ The proof proceeds by applying the Brown--Eagleson theorem to the triangular array \eqref{eq-martingale-array}. Conditions \eqref{eq-condition-incresing-chain} and \eqref{eq-condition-martingale} are clearly met, so it suffices to check conditions (C1)-(C3) there.

		By Lemma \ref{lem-Yn}-\ref{item:reYn2} and \eqref{eq-bn-Lindberg},
		\begin{align*}
			m_N&=\frac{1}{\rho^2_N} \max_{1\leq n \leq N} \bigg(\frac{1-|\lambda|^2}{2} |b_n|^2+ \frac{{\rm Re}\,(\alpha^2 \mu \,b_n^2f^{\circ (n+1)})}{2}\bigg)\\
			&\leq \frac{\max_{1\leq n\leq N} |b_n|^2}{\rho^2_N}\to 0, \quad \text{as $N\to \infty$}.
		\end{align*}
		This implies (C1). Moreover, we have 
		\begin{align} 
			V_N^2&= \frac{1}{\rho^2_N}\sum_{n=1}^N \bigg(\frac{1-|\lambda|^2}{2} |b_n|^2+ \frac{{\rm Re}\,(\alpha^2 \mu \,b_n^2f^{\circ (n+1)})}{2}\bigg) \nonumber \\
			&=\frac{1-|\lambda|^2}{2}+ \frac{1}{\rho^2_N}\sum_{n=1}^N \bigg(\frac{{\rm Re}\,(\alpha^2 \mu \,b_n^2f^{\circ (n+1)})}{2}\bigg). \label{eq-VN-value}
		\end{align}
		An application of Lemma \ref{lem-weak-law} implies $$\frac{1}{\rho^2_N} \sum_{n=1}^N b_n^2 f^{\circ n} \stackrel{p}{\longrightarrow} 0,$$ and hence the second term of $V_N^2$ converges to zero in probability. This shows (C2) holds true with $\sigma^2 = \frac{1-|\lambda|^2}{2}$. It remains to prove (C3). Note that $$\max_{1\leq n\leq N} \frac{{\rm Re}\, (\alpha b_n Y_n)}{\rho_N} \leq  \big(1+|\lambda|\big)\frac{\max_{1\leq n\leq N}|b_n|}{\rho_N}\to 0, \quad N\to \infty.$$ So for any fixed $\varepsilon>0$, we have
		\begin{align*}
			&\sum_{n=1}^N \E[X_{N,n}^2 I(|X_{N,n}|>\varepsilon)| \cal F_{N,n-1}]\\
			=&\frac{1}{\rho^2_N} \sum_{n=1}^N \E\bigg[\bigg({\rm Re}\, (\alpha b_n Y_n)\bigg)^2 I\bigg(\frac{{\rm Re}\, (\alpha b_n Y_n)}{\rho_N}> \varepsilon \bigg) \bigg| \F_{n+1}\bigg]\\
			\leq& \frac{1}{\rho^2_N} \sum_{n=1}^N \E\bigg[\bigg({\rm Re}\, (\alpha b_n Y_n)\bigg)^2 I\bigg(\max_{1\leq n\leq N}\frac{{\rm Re}\, (\alpha b_n Y_n)}{\rho_N}> \varepsilon\bigg) \bigg| \F_{n+1}\bigg]\\
			=&0
		\end{align*}
		provided that $N$ is sufficiently large. This proves (C3) and completes the proof of Theorem \ref{thm-NS}.
\end{proof}

	\noindent \textbf{Remark.} Parallel to Theorem \ref{thm-NS}, Nicolau and Soler i Gibert also established a CLT for tails: if $a\in \ell^2$ and $$\lim_{N\to \infty} \frac{|a_N|^2}{\sum_{n=N}^\infty |a_n|^2}=0,$$ then with $\sigma(N):= \big\|\sum_{n=N}^\infty a_n f^{\circ n}\big\|_2$, it holds that $$\frac{1}{\sigma(N)}\sum_{n=N}^\infty a_n f^{\circ n}\stackrel{d}{\longrightarrow} \mathcal{CN}(0,1),\quad \text{ as $N\to \infty$}.$$ The arguments used here can also be adapted to prove this result. In this case, the proof is even simpler: it follows directly from the central limit theorem for reverse martingales \cite{Loy69}, without requiring the triangular array version.

	\section{Proof of Theorem \ref{thm-BE-inner}}

We apply the same strategy as in the proof of Theorem \ref{thm-NS}, substituting the Brown--Eagleson Theorem with a deeper result of E. Haeusler \cite[Theorem 1]{Hae88}. The latter asserts that: for any $\delta>0$, there exists a constant $C$ depending only on $\delta$, such that
\begin{align*}
	&\sup_{x\in \R}\bigg| \PP\Big( \sum_{n=1}^N X_n\leq x \Big) - \PP \Big(\cal N(0,1) \leq x \Big)\bigg| \\
	\leq & C\cdot  \bigg( \sum_{n=1}^N \E\big[ |X_n|^{2+2\delta}\big]+ \E\bigg[\, \bigg|\sum_{n=1}^N \E\big[ X_n^2 | \cal F_{n-1}\big] -1\bigg|^{1+\delta}\bigg]\bigg)^{\frac{1}{3+2\delta}}
\end{align*}
holds for any real-valued martingale difference sequence $(X_n)_{n=1}^N$ of length $N$ (with respect to the filtration $(\cal F_n)^N_{n=0}$).

We also need the following lemma. 
\begin{lemma}\label{lem-normal-ball-probability}
	Let $\Phi(x)$ be the distribution function of $\mathcal N(0,1/2)$. Then
	$$
	\sup_{x\in \R}\bigg| \Phi\big(px+q\big) - \Phi\big(x\big)\bigg|
	\leq 1.35 \cdot (|p-1|+|q|),
	$$
	for any $1/2\leq p\leq 3/2$ and $-1\leq q \leq 1$.
\end{lemma}

\begin{proof} 
	By the mean value theorem, for any $x\in \R$, there exists $\xi$ between $x$ and $px+q$, such that
	$$
	\Big|\Phi\big(px+q\big)- \Phi(x)\Big|
	= \frac{|(p-1)x+q|}{\sqrt{\pi}} \,e^{-\xi^2}.
	$$
	Observe that
	$$
	|\xi|\geq \inf \big\{|px+q|: 1/2 \leq p \leq 3/2,\ -1\leq q \leq 1\big\}
	=\max\big\{0,|x|/2 -1 \big\}=: m(x),
	$$
	and hence
	$$
	C:= \frac{1}{\sqrt{\pi}}\max \bigg\{ \sup_{x\in \R} |x|\cdot e^{-m(x)^2}, \, 1\bigg\}\approx 1.348... <1.35 
	$$
	provides a desired upper bound.
\end{proof}

\begin{proof}[Proof of Theorem \ref{thm-BE-inner}]
We use the same notation as in the proof of Theorem \ref{thm-NS}. For any fixed $\delta>0$ and $\alpha\in \T$, set 
$$X_n:= \bigg(\frac{1-|\lambda|^2}{2}\bigg)^{-1/2} \cdot {\rm Re}\, \bigg( \frac{\alpha b_{N+1-n} Y_{N+1-n}}{\rho_N}\bigg), \quad 1\leq n \leq N.$$ Then we have
\begin{align}\label{eq-2+2delta-norm}
	\sum_{n=1}^N \E\big[ |X_n|^{2+2\delta}\big]\lesssim_{\delta,\lambda} \frac{\sum_{n=1}^N |b_n|^{2+2\delta}}{\rho_N^{2+2\delta}}
	\lesssim_{\delta,\lambda} \bigg( \sum_{n=1}^N |a_n|^{2+2\delta}\bigg) \bigg( \sum_{n=1}^N |a_n|^2\bigg)^{-1-\delta}.
\end{align}
Moreover, dividing both sides of \eqref{eq-VN-value} by $(1-|\lambda|^2)/2$, we obtain
\begin{align*} 
	\sum_{n=1}^N \E\big[ X_n^2 | \cal F_{n-1}\big]= 1+ \bigg(\frac{1-|\lambda|^2}{2}\bigg)^{-1} \frac{1}{\rho^2_N} \sum_{n=1}^N \bigg(\frac{{\rm Re}\,(\alpha^2 \mu \,b_n^2f^{\circ (n+1)})}{2}\bigg).
\end{align*}
This implies 
\begin{align*}
	&\bigg|\sum_{n=1}^N \E\big[ X_n^2 | \cal F_{n-1}\big] -1\bigg|^{1+\delta}\lesssim_{\delta,\lambda} \bigg| \frac{1}{\rho_N^2}\sum_{n=1}^N b_n^2 f^{\circ (n+1)}\bigg|^{1+\delta}.
\end{align*}
Now taking the expectation and using a Khintchine type inequality \cite[Corollary 1.4]{Nic22}, we get
\begin{align*} 
	\E\bigg[\, \bigg|\sum_{n=1}^N \E\big[ X_n^2 | \cal F_{n-1}\big] -1\bigg|^{1+\delta}\bigg] \lesssim_{\delta,f}& \bigg(\sum_{n=1}^N |b_n|^4 \bigg)^{\frac{1+\delta}{2}} \rho_N^{-2-2\delta} \\ 
	\lesssim_{\delta,f} &\bigg( \sum_{n=1}^N |a_n|^4\bigg)^{\frac{1+\delta}{2}}\bigg( \sum_{n=1}^N |a_n|^2\bigg)^{-1-\delta}
\end{align*}

Combining this with \eqref{eq-2+2delta-norm}, we deduce from Haeusler's theorem that
\begin{align}\label{eq-using-Haeusler}
\sup_{x\in \R}\bigg| \PP\Big( \sum_{n=1}^N X_n\leq x \Big) - \PP \Big(\cal N(0,1) \leq x \Big)\bigg| 
\lesssim_{\delta,f} \bigg(\sum_{n=1}^N |a_n|^{4} \bigg)^{\frac{1+\delta}{6+4\delta}}  \bigg(\sum_{n=1}^N |a_n|^{2} \bigg)^{-\frac{1+\delta}{3+2\delta}}
\end{align} whenever $\delta\geq 1$.

Moreover, by the identity \eqref{eq-sum-to-martingale}, standard calculations show that
\begin{align*} 
	\sum_{n=1}^N X_n = \sqrt{2} \cdot \bigg( p_N {\rm Re}\,\big( \alpha Z_N \big)- R_N\bigg),
\end{align*}
where $$p_N: =\frac{\sigma_N}{{\sqrt{1-|\lambda|^2}\rho_N}}$$ and the remainder $R_N$ satisfies $$|R_N| \leq \frac{|\lambda|}{\sqrt{1-|\lambda|^2}} \frac{|b_N|}{\rho_N} =: q_N.$$ Consequently, the left-hand side of \eqref{eq-using-Haeusler} can be rewritten as
$$\sup_{x\in \R} \bigg| \PP\bigg( {\rm Re}\, \big( \alpha Z_N\big)\leq \frac{x-R_N}{p_N} \bigg) - \PP\bigg( \cal N(0,1/2) \leq x \bigg)\bigg|.$$

Since $\lim_{N\to \infty}p_N=1$, we see for sufficiently large $N$, 
$$\PP\bigg( {\rm Re}\, \big( \alpha Z_N\big)\leq \frac{x-q_N}{p_N} \bigg) \leq \PP\bigg( {\rm Re}\, \big( \alpha Z_N\big)\leq \frac{x-R_N}{p_N} \bigg) \leq \PP\bigg( {\rm Re}\, \big( \alpha Z_N\big)\leq \frac{x+q_N}{p_N} \bigg).$$ 

Now applying Lemma \ref{lem-normal-ball-probability}, we conclude that 
\begin{align*} 
	& \sup_{x\in \R} \bigg[ \PP\bigg( {\rm Re}\, \big( \alpha Z_N\big)\leq x \bigg)- \PP\bigg( \cal N(0,1/2) \leq x \bigg) \bigg] \\
	= & \sup_{x\in \R} \bigg[ \PP\bigg( {\rm Re}\, \big( \alpha Z_N\big)\leq \frac{x-q_N}{p_N} \bigg)- \PP\bigg( \cal N(0,1/2) \leq \frac{x-q_N}{p_N} \bigg) \bigg] \\
	\leq &  \sup_{x\in \R} \bigg[ \PP\bigg( {\rm Re}\, \big( \alpha Z_N\big)\leq \frac{x-R_N}{p_N} \bigg)- \PP\bigg( \cal N(0,1/2) \leq \frac{x-q_N}{p_N} \bigg) \bigg]\\
	\leq  & \sup_{x\in \R} \bigg[ \PP\bigg( {\rm Re}\, \big( \alpha Z_N\big)\leq \frac{x-R_N}{p_N} \bigg)- \PP\bigg( \cal N(0,1/2) \leq x \bigg) \bigg] \\
	& + \sup_{x\in \R} \bigg[ \PP\bigg( \cal N(0,1/2) \leq x \bigg)- \PP\bigg( \cal N(0,1/2) \leq \frac{x-q_N}{p_N} \bigg) \bigg].
\end{align*}
The first supremum can be bounded using \eqref{eq-using-Haeusler}. Next we consider the second one. By \eqref{eq-rhoN-sigmaN} we have $\sigma_N^2 = (1-|\lambda|^2) \rho^2_N +|\lambda|^2 |b_N|^2,$ and hence $$|p_N-1|\leq \frac{|\lambda|}{\sqrt{1-|\lambda|^2}} \frac{|b_N|}{\rho_N}.$$ Now an application of Lemma \ref{lem-normal-ball-probability} gives 
\begin{align*} 
	& \sup_{x\in \R} \bigg[ \PP\bigg( \cal N(0,1/2) \leq x \bigg)- \PP\bigg( \cal N(0,1/2) \leq \frac{x-q_N}{p_N} \bigg) \bigg] \\
	 =& \sup_{x\in \R} \bigg[ \PP\bigg( \cal N(0,1/2) \leq p_Nx+q_N \bigg)- \PP\bigg( \cal N(0,1/2) \leq x \bigg) \bigg]\\
	 \leq & 1.35\cdot (|p_N-1|+|q_N|) \\
	 \leq & 2.7\cdot \frac{|\lambda|}{\sqrt{1-|\lambda|^2}} \frac{|b_N|}{\rho_N}.
\end{align*}
Then we can apply \eqref{eq-2-norm-eqvalence} and \eqref{eq-rhoN-sigmaN} to obtain
\[
 \sup_{x\in \R} \bigg[ \PP\bigg( \cal N(0,1/2) \leq x \bigg)- \PP\bigg( \cal N(0,1/2) \leq \frac{x-q_N}{p_N} \bigg) \bigg]\leq 2.7\cdot \frac{|\lambda|}{1-|\lambda|} \frac{|b_N|}{\sqrt{\sum_{n=1}^N |a_n|^2}}.
\]
Combining these estimates yields
\begin{align*} 
	&\sup_{x\in \R} \bigg[ \PP\Big( {\rm Re}\, \big( \alpha Z_N\big) \leq x \Big) - \PP\Big( \cal{N}(0,1/2) \leq x \Big)\bigg] \\
	\leq & \,\,C(\delta,f) \cdot \bigg(\sum_{n=1}^N |a_n|^{4} \bigg)^{\frac{1+\delta}{6+4\delta}}  \bigg(\sum_{n=1}^N |a_n|^{2} \bigg)^{-\frac{1+\delta}{3+2\delta}} + 2.7 \cdot \frac{|\lambda|}{1-|\lambda|} \frac{\big|\sum_{n=1}^N \lambda^{N-n} a_n\big|}{\sqrt{\sum_{n=1}^N |a_n|^2}}.
\end{align*}
The same upper bound for $$\sup_{x\in \R} \bigg[ \PP\Big( \cal{N}(0,1/2) \leq x \Big)-\PP\Big( {\rm Re}\, \big( \alpha Z_N\big) \leq x \Big) \bigg]$$ can be obtained by a similar argument, and the proof of Theorem \ref{thm-BE-inner} is completed.
\end{proof}

\section{Proofs of Lemma \ref{lem-Yn} and Lemma \ref{lem-transfer}}

	\begin{proof}[Proof of Lemma \ref{lem-Yn}] Clearly, $Y_n=f^{\circ n} - \lambda f^{\circ (n+1)}$ is $\F_{n}$-measurable. Furthermore, since $f^{\circ n}$ preserves the Lebesgue measure, one has
			\begin{align}
				\big\la f^{\circ n}, \big(f^{\circ (n+1)}\big)^k \big\ra = \big\la z \circ f^{\circ n}, f^k \circ f^{\circ n}\big\ra= \big\la z, f^k\big\ra =
				\begin{cases}
					\lambda\,, &  \text{if }k=1; \\ 
					 0\,, &  \text{otherwise}.
				\end{cases}
		\end{align}
		Thus $$\E\big[ f^{\circ n} | \F_{n+1}\big] = \sum_{k\in \mathbb Z} \big\la f^{\circ n}, \big(f^{\circ (n+1)}\big)^k \big\ra \cdot \big(f^{\circ (n+1)}\big)^k = \lambda f^{\circ (n+1)}.$$ This shows $(Y_n)_{n=1}^\infty$ is indeed a reverse martingale difference sequence. 

		Since $\overline{f^{\circ (n+1)}}$ is $\mathscr{F}_{n+1}$-measurable, we have $$\E\big[ f^{\circ n} \ol{f^{\circ (n+1)}} | \F_{n+1}\big] = \ol{f^{\circ (n+1)}} \cdot \E\big[ f^{\circ n} | \F_{n+1}\big] = \ol{f^{\circ (n+1)}}\cdot \lambda f^{\circ (n+1)} = \lambda.$$ This implies that 
		\begin{align*} 
			\E\big[|Y_n|^2 | \F_{n+1}\big] =& \E\big[ 1+ |\lambda|^2 - \lambda \ol{f^{\circ n}} f^{\circ (n+1)} - \ol{\lambda} f^{\circ n} \ol{f^{\circ (n+1)}} | \F_{n+1}\big]\\
			=&1+ |\lambda|^2  - \lambda \E\big[ \,\ol{f^{\circ n}} f^{\circ (n+1)}|\F_{n+1}\big] - \ol{\lambda} \E\big[f^{\circ n} \ol{f^{\circ (n+1)}} | \F_{n+1}\big]\\
			=& 1-|\lambda|^2.
		\end{align*}
		The identity in \ref{item:|Yn|2} follows.

		By expanding the square, we obtain $$Y_n^2 = \big(f^{\circ n}\big)^2- 2 \lambda f^{\circ n} f^{\circ (n+1)} + \lambda^2 \big(f^{\circ (n+1)}\big)^2.$$ 
		Note that the conditional expectation of the cross term is $$\E\big[2 \lambda f^{\circ n} f^{\circ (n+1)} | \F_{n+1}\big] =  2 \lambda^2  \big(f^{\circ (n+1)}\big)^2 .$$
		Moreover, it holds that 
		\begin{align*} 
			\big\la \big(f^{\circ n}\big)^2 , \big(f^{\circ (n+1)}\big)^k \big\ra = \big\la z^2 \circ f^{\circ n}, f^k \circ f^{\circ n}\big\ra= \big\la z^2 , f^k\big\ra =
				\begin{cases}
					\mu\, , &   \text{if }k=1; \\ 
					 \lambda^2 \, , &  \text{if }k=2;\\
					 0\, , & \text{otherwise}.
				\end{cases}
		\end{align*}
		This implies 
		\begin{align*} 
			\E\big[ \big(f^{\circ n}\big)^2  | \F_{n+1}\big] =& \sum_{k\in \mathbb Z} \big\la \big( f^{\circ n}\big)^2 , \big(f^{\circ (n+1)}\big)^k \big\ra \cdot \big(f^{\circ (n+1)}\big)^k\\ 
			=&\mu f^{\circ (n+1)} + \lambda^2 \big(f^{\circ (n+1)}\big)^2.
	\end{align*}
	Now \ref{item:Yn2} follows by combining these identities.
	
	The identity in \ref{item:reYn2} is a direct consequence of \ref{item:|Yn|2} and \ref{item:Yn2}. Indeed, observe that $$|\alpha Y_n|^2 = \big( {\rm Re}\, \alpha Y_n \big)^2 + \big({\rm Im}\, \alpha Y_n \big)^2, \quad \text{and}\quad {\rm Re}\, (\alpha Y_n)^2 = \big( {\rm Re}\, \alpha Y_n \big)^2 - \big({\rm Im}\, \alpha Y_n \big)^2.$$
	By taking the conditional expectation, we get 
	\begin{align*} 
		2 \E\big[({\rm Re}\, \alpha Y_n)^2 |\mathscr F_{n+1}\big]=& \E\big[ |\alpha Y_n|^2 | \F_{n+1}\big]+ \E\big[  {\rm Re}\, \big(\alpha Y_n \big)^2 | \F_{n+1}\big]\\
		=& |\alpha |^2 \,\E\big[ |Y_n|^2 | \F_{n+1}\big]+ {\rm Re}\, \big( \alpha^2 \, \E\big[ Y_n^2 | \F_{n+1}\big]\big)\\
		=& |\alpha|^2\,(1-|\lambda|^2)+ {\rm Re}\,(\alpha^2 \mu f^{\circ (n+1)}).
	\end{align*}
	The proof of Lemma \ref{lem-Yn} is completed.
	\end{proof}

		\begin{proof}[Proof of Lemma \ref{lem-transfer}]

		The claim \ref{item:keep-no-ell2} follows trivially from
		\begin{align*}
			a_1=b_1,\,\, \text{and} \,\, a_n=b_n-\lambda b_{n-1}, \text{ for }n\geq 2.
		\end{align*}
		Moreover, the claim \ref{item:keep-ell2} can be deduced by taking the $L^2(\T)$ norm on both sides of the following identity:
		$$\sum_{n=1}^\infty a_n z^n = (1-\lambda z) \cdot \bigg( \sum_{n=1}^\infty b_nz^n \bigg).$$
		Next we prove the claim \ref{item:keep-Lindberg}. Consider the generating functions $$A_N(z):=\sum_{n=1}^N a_n z^n,\quad \text{and} \quad B_N(z):=\sum_{n=1}^N b_n z^n.$$
		Then for $N\geq 2$ we have 
		\begin{align*}
			A_N(z)=(1-\lambda z) B_N(z) + \lambda b_Nz^{N+1}.
		\end{align*}
		This implies 
		\begin{align*}
		\big\|A_N \big\|_{L^2(\T)} \leq  \big(1+|\lambda|\big)\big\|B_N \big\|_{L^2(\T)}+|\lambda| |b_N|.
	\end{align*}
	
   Substituting $$|b_N| \leq \sum_{n=1}^N |\lambda|^{N-n} |a_n| \leq  \frac{\max_{1\leq n \leq N} |a_n|}{1-|\lambda|},$$ and using Parseval's identity, one has
	\begin{align}\label{eq1.2}
		\sqrt{\sum_{n=1}^N |a_n|^2} \leq (1+|\lambda|) \sqrt{\sum_{n=1}^N |b_n|^2} + \frac{ \lambda }{1-|\lambda|}\max_{1\leq n \leq N} |a_n|.
	\end{align}
	
	Now if $$\lim_{N\to \infty} \frac{\max_{1\leq n \leq N} |a_n|}{\sqrt{\sum_{n=1}^N |a_n|^2}} = 0,$$ then from \eqref{eq1.2} we deduce that 
	\begin{align}\label{eq-ratio-lower-bound}
		\liminf_{N\to \infty} \frac{\sqrt{\sum_{n=1}^N |b_n|^2}}{\sqrt{\sum_{n=1}^N |a_n|^2}}\geq \frac{1}{1+|\lambda|}.
	\end{align}
	This implies $$\limsup_{N\to \infty} \frac{|b_N|}{\sqrt{\sum_{n=1}^N |b_n|^2}}\leq \limsup_{N\to \infty}\frac{1+|\lambda|}{1-|\lambda|} \cdot \frac{\max_{1\leq n \leq N} |a_n|}{\sqrt{\sum_{n=1}^N |a_n|^2}} =0.$$ The claim \ref{item:keep-Lindberg} follows.
	\end{proof}



\end{document}